\documentclass[11pt]{article}
\usepackage{graphicx}
\usepackage{subfigure}
\usepackage{amssymb,bm,amsthm}
\usepackage{amsmath}
\usepackage{mathtools}
\usepackage{epstopdf,color}
\usepackage{booktabs}

\newtheorem{theorem}{Theorem}[section]

\newcommand{\md}{\mathrm{d}}
\allowdisplaybreaks[4]

\title{A Fast Algorithm for the Moments of Bingham Distribution}
\author{Yixiang Luo$^1$, Jie Xu$^2$ and Pingwen Zhang$^3$\\
\small
$^1$Courant Institute of Mathematical Sciences, New York University, New York 10012, USA\\
\small
$^2$Department of Mathematics, Purdue University, West Lafayette 47907, USA
\\
\small
$^3$LMAM \& School of Mathematical Sciences, Peking University, Beijing 100871, China\\
\small Email: yl4507@cims.nyu.edu, xu924@purdue.edu, pzhang@pku.edu.cn}
\date{\today}

\begin{document}

\maketitle

\begin{abstract}
We propose a fast algorithm for evaluating the moments of Bingham distribution.
The calculation is done by piecewise rational approximation, where interpolation and Gaussian integrals are utilized. Numerical test shows that the algorithm reaches the maximal absolute error less than $5\times 10^{-8}$ remarkably faster than adaptive numerical quadrature.
We apply the algorithm to a model for liquid crystals with the Bingham distribution to examine the defect patterns of rod-like molecules confined in a sphere, and find a different pattern from the Landau-de Gennes theory.

\vspace{12pt}
\end{abstract}

\section{Introduction}
The Bingham distribution is an important antipodally symmetric distribution on the unit sphere $\mathbb{S}^2$.
Although introduced from a statistical perspective \cite{bingham1974antipodally}, it has found applications in liquid crystals \cite{grosso2000closure,feng1998closure,ball2010nematic,han2014microscopic}, palaeomagnetism \cite{onstott1980application,kirschvink1980least,kent1983linear}, and various other fields involving data on the sphere \cite{descoteaux2009deterministic,alastrue2010use,kunze2004bingham,minka2000automatic,zhao2010fragment}. 

The density function of the Bingham distribution is given by
\begin{equation}\label{densOrg}
f(\bm{x} | B) = \left.{\exp \left( \sum_{i,j=1}^3B_{ij}x_ix_j \right)}
\middle /
{\int_{\mathbb{S}^{2}} \exp \left( \sum_{i,j=1}^3B_{ij}x_ix_j \right) \, \md \bm{x} }\right., \quad \bm{x}\in \mathbb{S}^2,
\end{equation}
where $B$ is a $3\times 3$ symmetric matrix.
A fundamental problem in computation involving the Bingham distribution is evaluating the moments
\begin{equation}\label{moments}
  \langle x_1^{n_1}x_2^{n_2}x_3^{n_3} \rangle=\int_{\mathbb{S}^{2}} f(\bm{x} | B) x_1^{n_1}x_2^{n_2}x_3^{n_3}\md\bm{x}.
\end{equation}
Denote
\begin{equation}\label{Z_Calc}
Z_{n_1n_2n_3}(B) = \int_{\mathbb{S}^{2}} x_1^{n_1} x_2^{n_2}x_3^{n_3} {\exp \left( \sum_{i,j=1}^3B_{ij}x_ix_j \right)} \, \md \bm{x}.
\end{equation}
Then the moments can be expressed as $\langle x_1^{n_1}x_2^{n_2}x_3^{n_3} \rangle=
Z_{n_1n_2n_3}(B)/Z_{000}(B)$.

Even when solving a single problem, the evaluation of moments \eqref{moments} may need to be done repeatedly. 
This is a typical case in the simulations of liquid crystals. 
In each iteration or time step, \eqref{moments} is computed at each grid point.
Generally speaking, the number of space discretization is $O(N^3)$. 
If we calculate \eqref{moments} by direct numerical quadrature, 
it costs $O(N^2)$ operations for every single calculations, 
leading to a total cost of $O(N^5)$. 
On the other hand, it should be noted that the density function \eqref{densOrg} is determined only by $B$, not relevant to parameters (and domains, etc.) specified by the problem to be solved. 
Therefore, it is desirable to have a fast algorithm for the evaluation of \eqref{moments}. 

The existing approximations of \eqref{moments} are designed only for special cases, and are not accurate enough to meet the demand of simulations in many problems. 
Kent \cite{kent1987asymptotic} proposed simple expansions for the zeroth and second moments. The relative error is about 0.1\%. Kume and Wood \cite{kume2005saddlepoint,kume2013saddlepoint} developed a method to compute the $Z_{000}(B)$ by using saddle-point approximation. It is accurate for the final estimation result when applying this method in doing maximum likelihood estimation, but not accurate enough for evaluating $Z_{000}(B)$. Moreover, the approximation cannot be easily extended to general $Z_{n_1n_2n_3}(B)$. 
Wang \textit{et. al.} \cite{wang2008crucial} used piecewise linear interpolation to compute $B$ from $Z_{n_1n_2n_3}/Z$ where $n_1+n_2+n_3=2$. This approach works well for $B$ not far from zero matrix, but is inaccurate when it is not the case. 
We also mention that in \cite{grosso2000closure} the fourth-order moments 
$Z_{n_1n_2n_3}/Z,\,(n_1+n_2+n_3=4)$, are approximated by polynomials of the second-order moments $Z_{n_1n_2n_3},\,(n_1+n_2+n_3=2)$, with a relative error of $5\times 10^{-4}$. 
This approach is restricted to the cases where $B$ is not involved explicitly. 

In this paper, we introduce a fast and accurate algorithm for evaluating $Z_{n_1n_2n_3}(B)$. 
We divide $B$ into three cases and use different approximation method for each case. 
The main techniques we utilize are interpolation and Gaussian integrals. 
We have implemented the method for $n_1+n_2+n_3\le 4$ in a routine named \texttt{BinghamMoments}. It is freely available online \cite{luo2016BinghamPack}, in which
pre-calculations are done and saved as constants in the routine to raise the real-time efficiency. 
The cost of evaluating $Z_{n_1n_2n_3}$ is reduced to $O(1)$ compared with $O(N^2)$ in numerical integration. 
Numerical experiments show that the absolute error is less than $5\times 10^{-8}$ in the routine, while $10^4$ times faster than adaptive numerical quadrature with the same accuracy. 
We apply the method to a liquid crystal model proposed in \cite{ball2010nematic,han2014microscopic}. 
The model substitutes the polynomial bulk energy in the widely-used Landau-de Gennes theory 
with the entropy term expressed by the Bingham distribution.
By this substitution the order parameters are confined in the physical range,
and it is shown in \cite{han2014microscopic} that this model can be derived from molecular theory.
We examine the defect patterns for rod-like molecules confined in a sphere, and find a different structure from the Landau-de Gennes theory.
The rest of paper is organized as follows. In Sec. \ref{method}, we present the approximation method. The numerical accuracy is examined in Sec. \ref{error}. An application to liquid crystals is given in Sec. \ref{appl}. Concluding remarks are stated in Sec. \ref{concl}.

\section{The approximation method\label{method}}
We diagonalize $B$ using an orthogonal matrix $T$ with $\mbox{det}T=1$,
$$
B=T\mbox{diag}(b_1,b_2,b_3)T^T.
$$
Then the density function becomes
\begin{equation}\label{dens_diag}
f(\bm{x}|B) = \left.\exp \left( \sum_{i=1}^{3} b_i (T^T\bm{x})_i^2 \right)
\middle /
\int_{\mathbb{S}^{2}} \exp \left( \sum_{i=1}^{3} b_i (T^T\bm{x})_i^2 \right) \, \md \bm{x}\right. .
\end{equation}
Thus, by the transformation $\bm{x}\longrightarrow T^T\bm{x}$,
\begin{align}
  Z_{n_1n_2n_3}(B)=&\int_{\mathbb{S}^{2}} x_1^{n_1} x_2^{n_2}x_3^{n_3} {\exp \left( \sum_{i=1}^{3} b_i (T^T\bm{x})_i^2 \right)} \, \md \bm{x} \nonumber\\
  =&\int_{\mathbb{S}^{2}} (T\bm{x})_1^{n_1} (T\bm{x})_2^{n_2}(T\bm{x})_3^{n_3} {\exp \left( \sum_{i=1}^{3} b_i x_i^2 \right)} \, \md \bm{x}
\end{align}
becomes a linear combination of $Z_{m_1m_2m_3}(\mbox{diag}(b_1,b_2,b_3))$.
Furthermore, the distribution $f(\bm{x}|\mbox{diag}(b_1,b_2,b_3))$ is invariant under changes $(b_1, b_2, b_3) \to (b_1 + h, b_2 + h, b_3 + h)$ for any real number $h$. Without loss of generality, we assume that $b_1 \le b_2 \le b_3 = 0$. Denote $Z_{n_1n_2n_3}(b_1,b_2)=Z_{n_1n_2n_3}(\mbox{diag}(b_1,b_2,0))$.
It is easy to note that $Z_{n_1n_2n_3}(b_1,b_2)$ is nonzero only if $n_i$ are even numbers. Then by $x_3^2=1-x_1^2-x_2^2$, we can express $Z_{n_1n_2n_3}(b_1,b_2)$ linearly by $Z_{nm0}(b_1,b_2)$. Hence it suffices to compute $Z_{nm0}(b_1,b_2)$, denoted in abbreviate by $Z_{nm}(b_1,b_2)$.

Choosing a parameter $d > 0$, we divide $(b_1,b_2)\in (-\infty,0]^2$ into three regions,
$$
(-\infty,-d]^2,\quad (-\infty,-d]\times(-d,0]\cup(-d,0]\times(-\infty,-d],\quad (-d,0]^2.
$$
and use different approximation method for each region.
The following Gaussian integral is used in the approximation,
\begin{equation}
  \int_{\mathbb{R}}x^{2n}\exp(-\alpha x^2)\md x= \sqrt{\frac{\pi}{\alpha}}\frac{(2n-1)!!}{(2\alpha)^n},\quad \alpha>0. \label{Gauss}
\end{equation}

\subsection{$b_1,\ b_2 \le -d$ \label{infinite_2}}
We transform the integral domain into the unit circle,
\begin{align}
Z_{nm}(b_1, b_2) = &2 \iint_{x_1^2+x_2^2 < 1} x_1^{n} x_2^{m} \cdot \exp \left( b_1 x_1^2 + b_2 x_2^2 \right) \cdot \frac{1}{ \sqrt{1 - x_1^2-x_2^2} } \, \md x_1\md x_2,\nonumber\\
=&2 \sum_{j,k\ge 0}{j+k \choose j}\frac{(2j+2k-1)!!}{(2j+2k)!!}\iint_{x_1^2+x_2^2 < 1}x_1^{2j+n}x_2^{2k+m} \exp \left( b_1 x_1^2 + b_2 x_2^2 \right)  \, \md x_1\md x_2. \label{platInte}
\end{align}
The series converges because $b_1,b_2<0$. We truncate the series at $j+k\le N_1$.
Moreover, if $d$ is large, then $x_1^{2j+n} x_2^{2k+m}$ increases with polynomial rate, while $ \exp \left( b_1 x_1^2 + b_2 x_2^2 \right)$ decreases with exponential rate.
Thus we expand the integral domain to $\mathbb{R}^2$ in the truncated series, which yields the following approximation formula,
\begin{align}
  \hat{Z}_{nm}(b_1, b_2) = &2 \sum_{j+k\le N_1}{j+k \choose j}\frac{(2j+2k-1)!!}{(2j+2k)!!}\iint_{\mathbb{R}^2}x_1^{2j+n}x_2^{2k+m} \exp \left( b_1 x_1^2 + b_2 x_2^2 \right)\, \md x_1\md x_2.\nonumber\\
  =&\sum_{j+k\le N_1}{j+k \choose j}\frac{(2j+2k-1)!!}{(2j+2k)!!}
  \sqrt{\frac{\pi^2}{b_1b_2}}\frac{(2j+n-1)!!(2k+m-1)!!}{(2b_1)^{j+n/2}(2b_2)^{k+m/2}}. \label{inf_approx1}
\end{align}

\subsection{$b_1 > -d$, $b_2 \le -d$ or $b_1 \le -d$, $b_2 > -d$ \label{infinite_1}}

We explain our approximation method by the case $b_1 \le -d$, $b_2 > -d$.
Rewrite $Z_{nm}( b_1, b_2)$ as
\begin{equation}\label{outCalaExpr}
Z_{nm}(b_1, b_2) = 4 \int_{-1}^{1} x_1^n \cdot \exp \left( b_1 x_1^2 \right) \cdot g_m( b_2, x_1) \, \md x_1,
\end{equation}
where
\begin{equation}
g_m( b_2, x_1) = \int_{0}^{\sqrt{1-x_1^2}} x_2^m \cdot \exp \left( b_2 x_2^2 \right) \cdot \frac{1}{ \sqrt{1 - x_1^2 - x_2^2} } \, \md x_2.
\end{equation}
Denote $a=1-x_1^2$ and $r=x_2/a$, then we have
\begin{align}
  g_m =& \int_{0}^{\sqrt{a}} x_2^m \cdot \exp \left( b_2 x_2^2 \right) \cdot \frac{1}{ \sqrt{a - x_2^2} } \, \md x_2\nonumber\\
  =& a^{m/2} \int_{0}^{1} r^m \cdot \exp \left( b_2 a r^2 \right) \cdot \frac{1}{ \sqrt{1 - r^2} } \, \md r \nonumber \\
  =& a^{m/2}  \cdot \frac{1}{2} \sqrt{\pi } \cdot \frac{\Gamma [(m+1)/2 ]}{\Gamma [ (m+2)/2 ]} \cdot {}_1F_1 \left( \frac{m+1}{2}; \frac{m+2}{2}; b_2 a \right), \nonumber
\end{align}
where 
$$
\Gamma(t)=\int_0^{\infty}x^{t-1}\exp(-x)\md x
$$
is the gamma function, and ${}_1F_1$ denotes the confluent hypergeometric function.

Note that ${}_1F_1(\frac{m+1}{2}; \frac{m+2}{2}; b_2 a )$ is an entire function about $a\in\mathbb{C}$.
Therefore $g_m(b_2,x_1)$ equals to its Taylor's series at $x_1=0$ for $x_1\in(-1,1)$,
$$
g_m( b_2, x_1) =\sum_{j\ge 0} \frac{1}{(2j)!} \left( \frac{\partial^{2j}}{\partial x_1^{2j}} g_m( b_2, 0) \right) x_1^{2j}.
$$

Similar to the case $b_1,b_2\le -d$, we truncate the series at $j\le N_2$. Again noticing $b_1\le -d$, we expand the integral interval in \eqref{outCalaExpr} to $\mathbb{R}$, leading to the approximation formula
\begin{align}
\hat{Z}_{nm}(b_1, b_2) =&4 \sum_{j\le N_2} \frac{1}{(2j)!} \left( \frac{\partial^{2j}}{\partial x_1^{2j}} g_m( b_2, 0) \right)\int_{\mathbb{R}} \exp \left( b_1 x_1^2 \right) x_1^{2j+n} \, \md x_1\nonumber\\
=&4 \sum_{j\le N_2} \frac{1}{(2j)!} \left( \frac{\partial^{2j}}{\partial x_1^{2j}} g_m( b_2, 0) \right)\cdot \sqrt{\frac{\pi}{-b_1}}\frac{(2j+n-1)!!}{(-2b_1)^{j+n/2}}\, .
\label{formulaAppr2}
\end{align}

Next, we explain how to calculate the derivatives ${\partial^{2j}g_m( b_2, 0)}/{\partial x_1^{2j}} $.
Denote
$$
h_1(a)=a^{m/2}, \;\; h_2(a)={}_1F_1 \left( \frac{m+1}{2}; \frac{m+2}{2}; b_2 a \right).
$$
Then we have
\begin{equation}
  \frac{\partial^j g}{\partial a^j} = \frac{1}{2} \sqrt{\pi } \cdot \frac{\Gamma [(m+1)/2 ]}{\Gamma [ (m+2)/2 ]} \cdot \sum_{k=0}^{j} \binom{j}{k} \partial^k_{a} h_1 \cdot \partial^{j-k}_{a} h_2,
\end{equation}
with
\begin{equation}
  \partial^k_{a} h_1=\frac{(m/2)!}{(m/2-k)!} a^{\frac{m}{2}-k}, \; k \leq \frac{m}{2}, \;\;\;\; \;\;\;\; \partial^k_{a} h_1=0, \;  k > \frac{m}{2},
\end{equation}
and
\begin{equation}
  \partial^k_{a} h_2 = b_2^k \left( \frac{m+1}{2} \right)^{(k)} \Big / \left( \frac{m+2}{2} \right)^{(k)} \cdot {}_1F_1 \left( \frac{m+1}{2}+k; \frac{m+2}{2}+k; b_2 a \right)
\end{equation}
where
$$x^{(0)}=1, \;\;\;\; x^{(k)}=x(x+1)(x+2) \cdots (x+k-1)$$
is the rising factorial. 
Along with
$$
\frac{\partial^2 a}{\partial x_1^2} \bigg |_{x_1=0}=-2, \;\;\;\; \frac{\partial^i a}{\partial x_1^i} \bigg |_{x_1=0}=0, \;\; i \neq 2,
$$
and the chain rule, we arrive at
\begin{equation}
  \frac{\partial^{2j}}{\partial x_1^{2j}} g(x_2^m | b_2, x_1) \, \bigg |_{x_1=0} = (-1)^{j} \cdot \frac{(2j)!}{j!} \cdot \frac{\partial^{j} g}{\partial a^{j}} \bigg |_{a=1}.
\end{equation}

The derivatives ${\partial^{2j}g_m( b_2, 0)}/{\partial x_1^{2j}} $ are functions of $b_2$. In the routine \texttt{BinghamMoments}, we precompute the values on grid points $b_2=0.001k$, and compute the values between the grid points by linear interpolation.

\subsection{$b_1 > -d$, $b_2 > -d$ \label{finite_2}}
In this bounded region of $(b_1,b_2)$, we use interpolation for $Z_{00}$ and $Z_{mn}/Z_{00}$.
We compute them and their derivatives about $b_1$, $b_2$,
$$
\frac{\partial Z_{00}}{\partial b_1}=Z_{20}, \quad
\frac{\partial (Z_{nm}/Z_{00})}{\partial b_1}=\frac{Z_{n+2,m}Z_{00}-Z_{nm}Z_{20}}{Z_{00}^2},
$$
on the grid
$(b_{1,2})_j=-j\Delta b,\ 0\le j\le -d/\Delta b$. 
These values are computed in advance and saved as constants in the routine \texttt{BinghamMoments}. 
For $Z_{nm}$ not on the grid points, we calculate with the interpolation described below.
Suppose we already know
$$
f(x_i,y_j), \;\; f_x(x_i,y_j), \;\; f_y(x_i,y_j),\quad j=1,2.
$$
To obtain the approxiamte value $f(x,y)$ on $(x,y) \in [x_1,x_2] \times [y_1,y_2]$, we first calculate
$$ f(x,y_1), \;\; f(x,y_2), \;\; f(x_1,y), \;\; f(x_2,y) $$
with third-order Hermite interpolation,
\begin{eqnarray}
f(x,y_1) &=& f(x_1,y_1) \cdot (1 + 2 \frac{x_1-x}{x_1-x_2}) (\frac{x-x_2}{x_1-x_2})^2 
+ f(x_2,y_1) \cdot (1 + 2 \frac{x_2-x}{x_2-x_1}) (\frac{x-x_1}{x_2-x_1})^2 \nonumber \\
&& + f_x(x_1,y_1) \cdot (x-x_1) (\frac{x-x_2}{x_1-x_2})^2 
+ f_x(x_2,y_1) \cdot (x-x_2) (\frac{x-x_1}{x_2-x_1})^2. \nonumber
\end{eqnarray}
Next we calculate
$$ f_y(x,y_1), \;\; f_y(x,y_2), \;\; f_x(x_1,y), \;\; f_x(x_2,y) $$
with linear interpolation,
$$ f_y(x,y_1) = f_y(x_1,y_1) \frac{x_2-x}{x_2-x_1} + f_y(x_2,y_1) \frac{x-x_1}{x_2-x_1}. $$
Then we can calculate $f(x,y)$ with third order Hermite interpolation by
\begin{eqnarray}
f(x,y) &=& f(x_1,y) \cdot (1 + 2 \frac{x_1-x}{x_1-x_2}) (\frac{x-x_2}{x_1-x_2})^2 
+ f(x_2,y) \cdot (1 + 2 \frac{x_2-x}{x_2-x_1}) (\frac{x-x_1}{x_2-x_1})^2 \nonumber \\
&& + f_x(x_1,y) \cdot (x-x_1) (\frac{x-x_2}{x_1-x_2})^2 
+ f_x(x_2,y) \cdot (x-x_2) (\frac{x-x_1}{x_2-x_1})^2, \label{interp1}
\end{eqnarray}
or
\begin{eqnarray}
f(x,y) &=& f(x,y_1) \cdot (1 + 2 \frac{y_1-y}{y_1-y_2}) (\frac{y-y_2}{y_1-y_2})^2 
+ f(x,y_2) \cdot (1 + 2 \frac{y_2-y}{y_2-y_1}) (\frac{y-y_1}{y_2-y_1})^2 \nonumber \\
&& + f_y(x,y_1) \cdot (y-y_1) (\frac{y-y_2}{y_1-y_2})^2 
+ f_y(x,y_2) \cdot (y-y_2) (\frac{y-y_1}{y_2-y_1})^2. \label{interp2}
\end{eqnarray}
We compute $f(x,y)$ as the average of \eqref{interp1} and \eqref{interp2}. 

\subsection{The value of the parameters}
We have introduced four parameters in the above: the size $d$ for dividing the domain, the order of truncation $N_1$ and $N_2$, and the grid size for the interpolation $\Delta b$.
We choose parameters as $d=30$, $N_1=5$, $N_2=5$, $\Delta b=0.025$ for $Z_{00}$, and $\Delta b=0.1$ for $Z_{nm}/Z_{00}$ in the routine \texttt{BinghamMoments}, achieving maximal absolute error less than $5\times 10^{-8}$ for $Z_{00}$ and $\langle x_1^n x_2^m \rangle,~ n+m\le 4$. We will verify this in Sec. \ref{Numerical Test}. With these parameters, the memory needed for loading precomputed values (including ${\partial^{2j}g_m( b_2, 0)}/{\partial x_1^{2j}} $ in the case \ref{infinite_1}, and the values on the grid points in the case \ref{finite_2}) is about 75MB, which is available for common computers. 

\section{Numerical accuracy\label{error}}
\subsection{Error estimate}
We give an error estimate for the case \ref{infinite_2} with some special functions.
Denote
$$
F(x)=e^{-x^2} \int_{0}^{x} e^{t^2} \, \md t
$$
as the Dawson function,
$$
\gamma(n,x)=\int_0^x t^{n-1}e^{-t}\, \md t
$$
as the lower incomplete gamma function, and
$$
\alpha_n(z)=E_{-n}(z)=n! z^{-n-1} e^{-z} \left ( 1+z+\frac{z^2}{2!}+ \cdots + \frac{z^n}{n!} \right )
$$
as the exponential integral function.

\begin{theorem}
Let $\hat{Z}_{nm}$ be defined in \eqref{inf_approx1} and denote $N=N_1$.
For $b_1,b_2 \le -d$, it holds
\begin{eqnarray} \label{infinite max error}
  |Z_{nm}-\hat{Z}_{nm}| &\leq& 4 \pi \frac{ F(\sqrt{d}\,)}{\sqrt{d}}- 2 \pi  \sum _{j=0}^N \frac{(2 j-1)!!}{(2 j)!!} \cdot d^{-j-1} \gamma (j+1,d) \nonumber \\ && + 2 \pi \sum_{j=0}^{N+\max(n,m)} \frac{(2j-1)!!}{(2j)!!} \alpha_j(d).
\end{eqnarray}
\end{theorem}
\begin{proof}
We can divide the error into two parts:
\begin{align}
 e_1=&  Z_{nm}( b_1, b_2) - 2 \iint_{B(0,1)} x_1^{n} x_2^{m} \exp \left( b_1 x_1^2 + b_2 x_2^2 \right) \sum_{j=0}^{N} \frac{(2j-1)!!}{(2j)!!} \cdot (x_1^2 + x_2^2)^j \, \md \bm{x} \nonumber\\
 =&2 \iint_{B(0,1)} x_1^{n} x_2^{m} \exp \left( b_1 x_1^2 + b_2 x_2^2 \right) \sum_{j>N} \frac{(2j-1)!!}{(2j)!!} (x_1^2 + x_2^2)^j \, \md \bm{x},\\
e_2=&  2 \iint_{\mathbb{R}^2 \setminus B(0,1)} x_1^{n} x_2^{m} \exp \left( b_1 x_1^2 + b_2 x_2^2 \right) \sum_{j=0}^{N} \frac{(2j-1)!!}{(2j)!!} (x_1^2 + x_2^2)^j \, \md \bm{x}.
\end{align}

For $e_1$, we have
\begin{eqnarray}
e_1 &\le&2 \iint_{B(0,1)} \exp \left( -d (x_1^2 + x_2^2) \right) \sum_{j>N} \frac{(2j-1)!!}{(2j)!!} (x_1^2 + x_2^2)^j \, \md \bm{x}\nonumber\\
&=&2 \iint_{B(0,1)} \exp \left( -d (x_1^2 + x_2^2) \right) \left[\frac{1}{\sqrt{1-x_1^2-x_2^2}}-\sum_{j\le N} \frac{(2j-1)!!}{(2j)!!} (x_1^2 + x_2^2)^j \right]\, \md \bm{x}\nonumber\\
&\le& 4 \pi \int_0^1 e^{-d r^2}  \frac{r}{1 - r^2} \, \md r - 4 \pi \sum_{j=0}^{N} \frac{(2j-1)!!}{(2j)!!} \int_0^1 r^{2j+1}  e^{-d r^2} \, \md r \nonumber \\
&=& 4 \pi \frac{ F(\sqrt{d}\,)}{\sqrt{d}}- 2 \pi  \sum _{j=0}^N \frac{(2 j-1)!!}{(2 j)!!}  d^{-j-1} \gamma (j+1,d). \label{err_1_inf}
\end{eqnarray}
In the above, we use the polar coordinate transformation
$x_1=r \cos \theta, \;\; x_2=r \sin \theta$.
For $e_2$, denote $M=\max\{ n, m \}$, then we have
\begin{eqnarray}
e_2 &\leq& 4 \pi \sum_{j=0}^{N+M} \frac{(2j-1)!!}{(2j)!!} \int_1^{\infty} r^{2j+1} e^{-d r^2} \, \md r = 2 \pi \sum_{j=0}^{N+M} \frac{(2j-1)!!}{(2j)!!} \alpha_n(d). \label{err_2_inf}
\end{eqnarray}

Combining \eqref{err_1_inf} and \eqref{err_2_inf}, we get \eqref{infinite max error}.
\end{proof}

\begin{table}
\centering
\begin{tabular}{ccccccc}
  \toprule
  $d$ & 13 & 16 & 20 & 26 \\
  \midrule
  $N_1$ & 5 & 6 & 6 & 6 \\
  \midrule
  Bound & $4.4\times 10^{-5}$ & $3.6\times 10^{-6}$ & $4.5\times 10^{-7}$ & $4.5\times 10^{-8}$ \\
  \bottomrule
\end{tabular}
\caption{Absolute error bound by \eqref{infinite max error} under different values of $d$ and $N_1$ for $n+m\le 4$. }
\label{estimated upper bound}
\end{table}
For our chosen parameters $d=30$ and $N_1=5$, the upper bound given by \eqref{infinite max error} is $6.038 \times 10^{-8}$ for $n+m\le 4$.
We also give the upper bound calculated from \eqref{infinite max error} for a few $d$ and $N_1$ in Table \ref{estimated upper bound}. 
The estimate \eqref{infinite max error} is also helpful to choosing parameters under different demand of accuracy, which will be shown in Table \ref{suggested value}.

\subsection{Numerical Test\label{Numerical Test}}
We compare the results calculated by our method and the results calculated by numerical integration to testify the accuracy of our method numrically.
The parameters in our method are chosen as $N_1=5$, $N_2=5$ and $d=30$.
For numerical integration, we use adaptive Simpson's method to control the absolute error less than $10^{-11}$. We select randomly $10,000$ pairs of $b_i$ for each of the three cases: $b_1,b_2\le -d$, $\max(b_1,b_2)>-d \; \& \; \min(b_1,b_2)\le -d$, and $b_1,b_2>-d$, and calculate $Z$ and the moments $Z_{nm}/Z$ where $n+m=2,4$. 
Table \ref{maxError} shows the maximal absolute errors of $Z$ and $Z_{nm}/Z$ among the $30000$ samples, 
which are under the magnitude of $10^{-8}$. 
In particular, the errors of $Z_{nm}/Z$ are less than $5\times 10^{-8}$. 
We also examine the distribution of the absolute errors of
$Z$ (Figure \ref{errorTest:a}), $Z_{20}/Z$ (Figure \ref{errorTest:b}) and $Z_{04}/Z$ (Figure \ref{errorTest:c})
for the $10000$ samples in each of three cases respectively, 
and find that for most $b_i$ the absolute errors are less than $10^{-10}$.
Moreover, the numerical test also shows our method is very fast.
Calculating all these $30,000$ examples, the adaptive Simpson's method with the target accuracy $5\times 10^{-8}$ spend $3117.761$ seconds while our method only $0.193$ seconds. 
Both routines are written in C and run in the same computer with a CPU clock speed $2.5$GHz. 

\begin{table}
\centering
\begin{tabular}{cccc}
  \toprule
  Moment & $Z$ & $Z_{20}/Z$ & $Z_{02}/Z$ \\
  \midrule
  Maximal error & $6.038\times 10^{-8}$ & $2.030\times 10^{-8}$ & $1.543 \times 10^{-8}$ \\
  \bottomrule \toprule
  Moment & $Z_{40}/Z$ & $Z_{04}/Z$ & $Z_{22}/Z$ \\
  \midrule
  Maximal error & $4.031\times 10^{-9}$ & $2.049\times 10^{-8}$ & $2.098\times 10^{-8}$ \\
  \bottomrule
\end{tabular}
\caption{Maximal absolute error for the $30,000$ pairs of $(b_1,b_2)$. }
\label{maxError}
\end{table}

\begin{figure}
\centering
\subfigure[absolute errors of $Z$] { \label{errorTest:a}
\includegraphics[width=0.3\columnwidth]{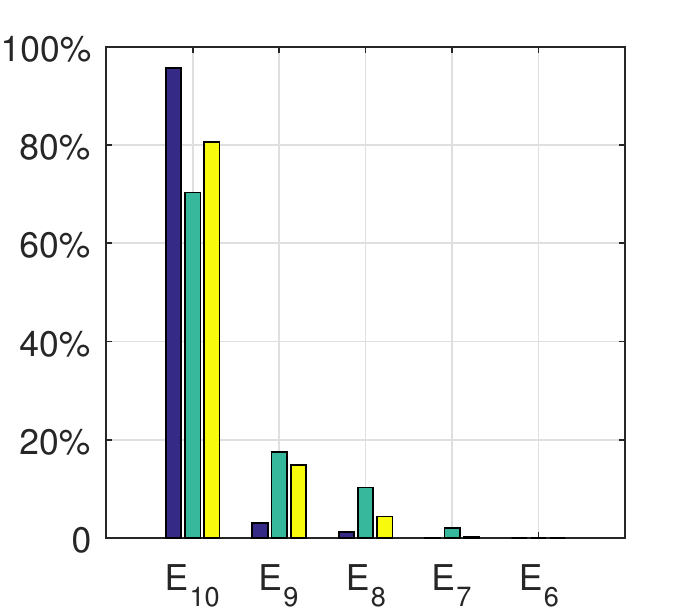}
}
\subfigure[absolute errors of $Z_{20}/Z$] { \label{errorTest:b}
\includegraphics[width=0.3\columnwidth]{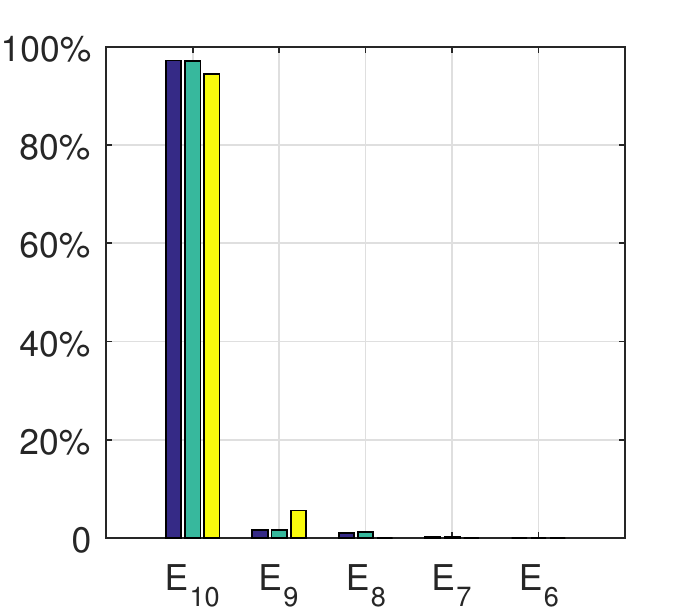}
}
\subfigure[absolute errors of $Z_{04}/Z$] { \label{errorTest:c}
\includegraphics[width=0.3\columnwidth]{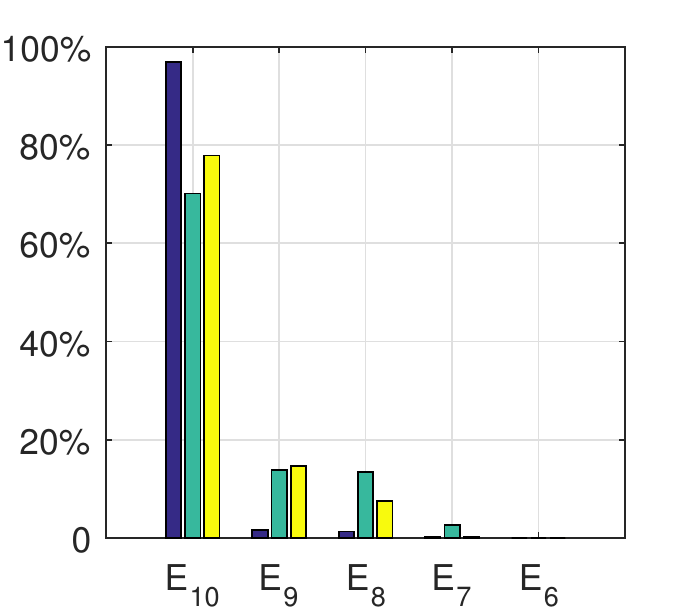}
}
\caption{Distribution of error. Blue bars: $b_1,b_2\le -d$; Green bars: $\max(b_1,b_2)>-d$ \& $\min(b_1,b_2)\le -d$; Yellow bars: $b_1,b_2>-d$. $E_i$ represents the interval $[10^{-i-1}, 10^{-i})$ for $i=7,8,9$. $E_{10}=[0, 10^{-10})$ and $E_6=[10^{-7}, \infty)$. }
\label{errorTest}
\end{figure}

We also give some other suggested values of $d$, $N_1$ and $N_2$ 
in Table \ref{suggested value} for different demanded accuracy for $Z_{nm}/Z$, 
which are also testified numerically with $30,000$ random samples. 
By comparing with the errors in Table \ref{suggested value} and Table \ref{estimated upper bound}, we find that the upper bound given by \eqref{infinite max error} are indicative for the choice of parameters. 
\begin{table}
\centering
\begin{tabular}{ccccccccc}
  \toprule
  Demanded maximal absolute error & $5\times 10^{-5}$ & $5\times 10^{-6}$ & $5\times 10^{-7}$ & $5\times 10^{-8}$ \\
  \midrule
  $d$ & 13 & 16 & 20 & 26 \\
  \midrule
  $N_1$ & 5 & 6 & 6 & 6 \\
  \midrule
  $N_2$ & 4 & 5 & 6 & 6 \\
  \bottomrule
\end{tabular}
\caption{Suggested values of parameters $d$, $N_1$ and $N_2$ under different demanded absolute error. }
\label{suggested value}
\end{table}

\section{Application to liquid crystals\label{appl}}

In this section, we apply our algorithm to a $Q$-tensor model for rod-like liquid crystals. Compared with the original Landau-de Gennes $Q$-tensor theory, the model is able to constrain the tensor within the physical range \cite{ball2010nematic}, and is closely connected to molecular theory \cite{han2014microscopic}. But the Bingham distribution in the model brings difficulty in numerical simulations. We will explain how our fast algorithm accelerates the computation. 

Suppose that the rod-like molecules are confined inside the unit sphere. 
Then the anchoring effect on the spherical surface will induce defects for the alignment of the molecules. 
We consider the following simplified free energy, 
\begin{equation}\label{freeEnergy}
F=\int_{\Omega} \md x\md y\md z\, \Big[ (B : (Q+\frac{I}{3}) - \log Z) - \frac{1}{2} \alpha_1 \left | Q \right |^2 + \frac{1}{2} \alpha_2 \left | \nabla  Q \right |^2\Big]  + F_{p},
\end{equation}
where the region $\Omega$ is chosen as the unit sphere, $I$ is the identity matrix, and
$$
Q_{ij}(\bm{x})=\int_{\mathbb{S}^2}(x_ix_j-\frac{1}{3}\delta_{ij})f(\bm{x}|B)\md S
$$
is a symmetric traceless matrix describing the orientational distribution of rod-like molecules at each spatial point, with $f(\bm{x}|B)$ and $Z=Z_{000}(B)$ defined in \eqref{densOrg} and \eqref{moments}. Here $\delta_{ij}$ is the Kronecker notation.
The first two terms in the integral are the bulk energy describing the nematic phase in equilibrium.
This bulk energy is the only terms distinct from the phenomenological Landau-de Gennes theory, where the bulk energy is given as a polynomial
$$
a_2\mbox{tr}(Q^2)-a_3\mbox{tr}(Q^3)+a_4(\mbox{tr}(Q^2))^2.
$$
The gradient term is the energy contribution of the spatial inhomogeneity. 
The boundary penalty term
$$ F_{p}=\int_{\partial\Omega}\md S\, [Q_{11} xy - Q_{12}(x^2-\frac{1}{3})]^2 + [Q_{12} z - Q_{13} y]^2 + [Q_{22} xy -Q_{12}(y^2-\frac{1}{3})]^2 + [Q_{12} z - Q_{23} x]^2 $$
is added to enforce the value of $Q$ on the sphere to be approximately 
$$
Q = \lambda \left(
\begin{array}{ccc}
  x^2-\frac{1}{3} & xy & xz\\
  xy & y^2-\frac{1}{3} & yz\\
  xz & yz & z^2-\frac{1}{3}
\end{array}
\right).
$$
In fact, if $Q$ is given as above, then $F_{p}=0$.
Our aim is to find local minimizers of the energy functional \eqref{freeEnergy} that describe metastable states.

Express $B$ as $B=T\mbox{diag}(b_1,b_2,0)T^T$, where $T$ is orthogonal with $\mbox{det}T=1$ and can be expressed by Euler angles,
$$
T=\left(
\begin{array}{ccc}
 \cos \alpha  \cos \gamma -\cos \beta  \sin \alpha  \sin \gamma  & \cos \gamma  \sin \alpha +\cos \alpha  \cos \beta  \sin \gamma  &
   \sin \beta  \sin \gamma  \\
 -\cos \beta  \cos \gamma  \sin \alpha -\cos \alpha  \sin \gamma  & \cos \alpha  \cos \beta  \cos \gamma -\sin \alpha  \sin \gamma  &
   \cos \gamma  \sin \beta  \\
 \sin \alpha  \sin \beta  & -\cos \alpha  \sin \beta  & \cos \beta  \\
\end{array}
\right).
$$
In this case, $Q=T\mbox{diag}(q_1,q_2,q_3)T^T$,
where the eigenvalues are given by $q_1=Z_{20}(b_1,b_2)/Z_{00}(b_1,b_2)$,
$q_2=Z_{02}(b_1,b_2)/Z_{00}(b_1,b_2)$, and $q_3=1-q_1-q_2$.

We use the spherical coordinates $(r,\theta,\phi)$ to represent the position, i.e.,
\begin{equation}
  x=r\sin\theta\cos\phi,\quad
  y=r\sin\theta\sin\phi,\quad
  z=r\cos\theta. \label{SphCoord}
\end{equation}
The integral becomes $\int (\cdot)\md x\md y\md z=\int (\cdot)r^2\sin\theta\md r\md\theta \md\phi$,
and the gradient term becomes
\begin{align}
|\nabla Q|^2=|\partial_r Q|^2+\frac{1}{r^2}|\partial_{\theta} Q|^2+\frac{1}{r^2\sin^2\theta}|\partial_{\phi} Q|^2.
\end{align}
The free energy is discretized at $N \times N \times N=32^3$ Gaussian quadrature nodes $(r_j,\theta_k,\phi_l)$ in $[0,1]\times[0,\pi]\times[0,2\pi]$.
At each node $(b_1,b_2,\alpha,\beta,\gamma)^{jkl}$ act as the basic variables,
from which $Q^{jkl}$ is computed.
The gradient term is computed using the spectral-collocation method.
From the value of $Q$ at the discretized nodes, a polynomial
$$
Q(r,\theta,\phi) = \sum_{j=0}^{N-1} \sum_{k=0}^{M-1} \sum_{l=0}^{L-1} c_Q^{jkl} \, r^j \theta ^k \phi ^l
$$
is constructed through interpolation.
The derivatives about $(r,\theta,\phi)$, as well as the values on the boundary,
are then computed from the above polynomial.
We refer to \cite{shen2011spectral} where the details about the spectral-collocation method are illustrated. 
The free energy is minimized using the BFGS method (see, for instance, \cite{avriel2003nonlinear}). 
In the iteration we need to compute the derivatives of $F$ about $(b_i)^{jkl}$,
where fourth moments are involved. For instance,
$$
\frac{\partial }{\partial b_1} Q = T\mbox{diag}(\frac{\partial q_1}{\partial b_1},\frac{\partial q_2}{\partial b_1},\frac{\partial (-q_1-q_2)}{\partial b_1})T^T,
$$
where
$$
\frac{\partial q_1}{\partial b_1}=\frac{Z_{40}Z_{00}-Z^2_{20}}{Z^2_{00}}.
$$
It is worth pointing out that at each point, the value of $Q$ and $Z$ are computed from $B$.
Therefore, our algorithm is executed $O(N^3)$ times in each BFGS iteration step, which greatly accelerates the simulation. 
Another thing is that the Bingham distribution remains the same 
when we alter the parameters $\alpha_{1,2}$, the domain (from sphere to cylinder or ellipse, etc.), and add some terms like in \cite{han2014microscopic}. 
Thus our algorithm is suitable for all these cases. 

Before looking at the results, we first define the biaxiality.
When $Q\neq 0$, we say $Q$ is uniaxial if it has two identical eigenvalues, and is biaxial if it has distinct eigenvalues. Note that $\mbox{tr}Q=0$. 
The biaxiality is measured by
$$
\mu = 1 - 6 \frac{(\mbox{tr}Q^3)^2}{(\mbox{tr}Q^2)^3}.
$$
For uniaxial $Q$, we have $\mu=0$; for biaxial $Q$, we have $0< \mu\le1$.
We examine the defect pattern under different $\alpha_1$ and $\alpha_2$.
At each point, the favored direction of the rod-like molecules is the principal
unit eigenvector $\bm{n}$ of $Q$. While $Q$ is continuous in the unit sphere,
$\bm{n}$ might be discontinuous at the points where $Q=0$ or $Q$ has two identical positive eigenvalues.
Defect patterns are classified by the configuration of these points.

\begin{figure} 
\centering
\subfigure[Radial hedgehog] { \label{LCPhase:a}
\includegraphics[width=0.3\columnwidth]{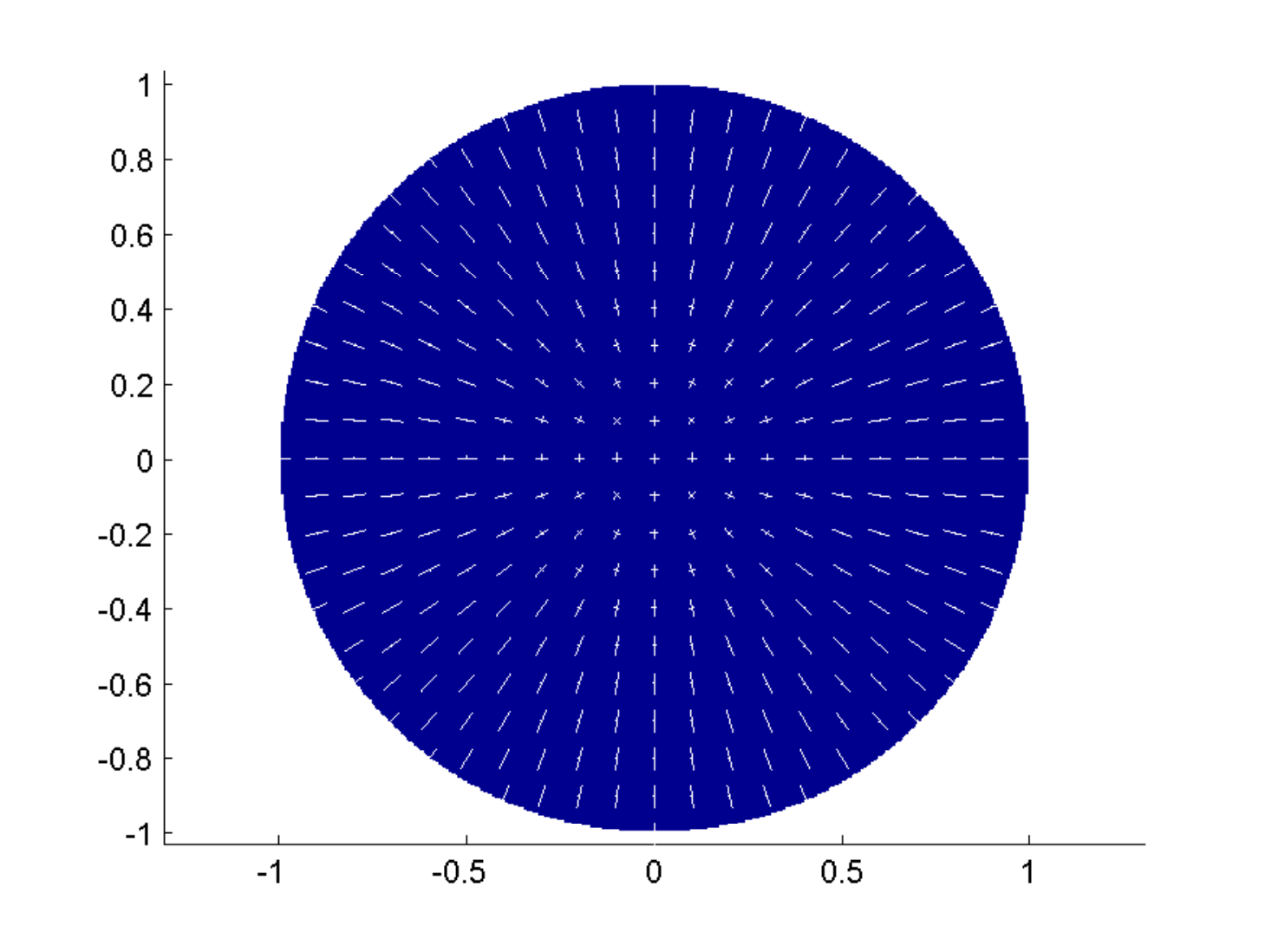}
}
\subfigure[Ring disclination] { \label{LCPhase:b}
\includegraphics[width=0.3\columnwidth]{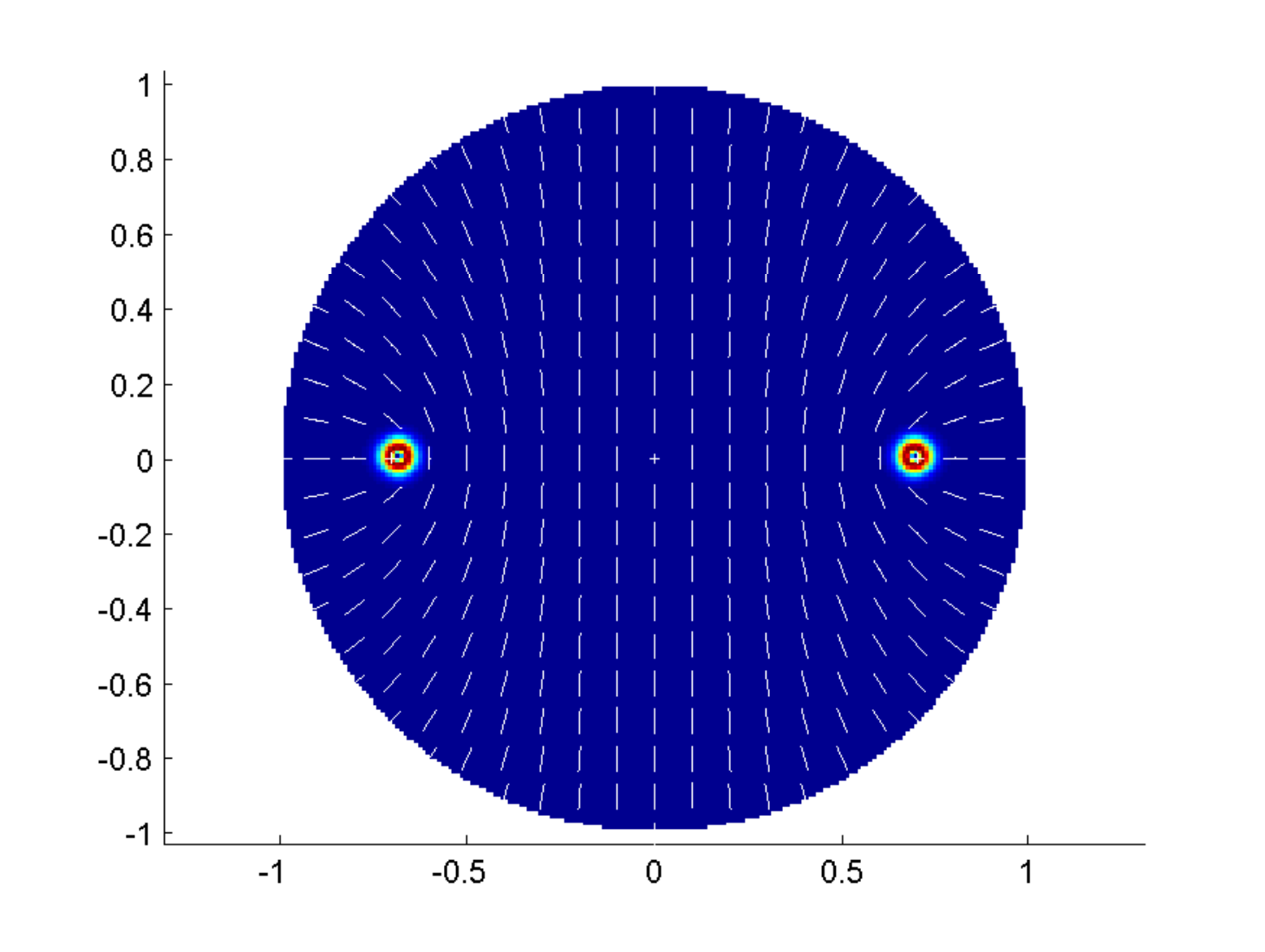}
}
\subfigure[Sphere ring band] { \label{LCPhase:c}
\includegraphics[width=0.3\columnwidth]{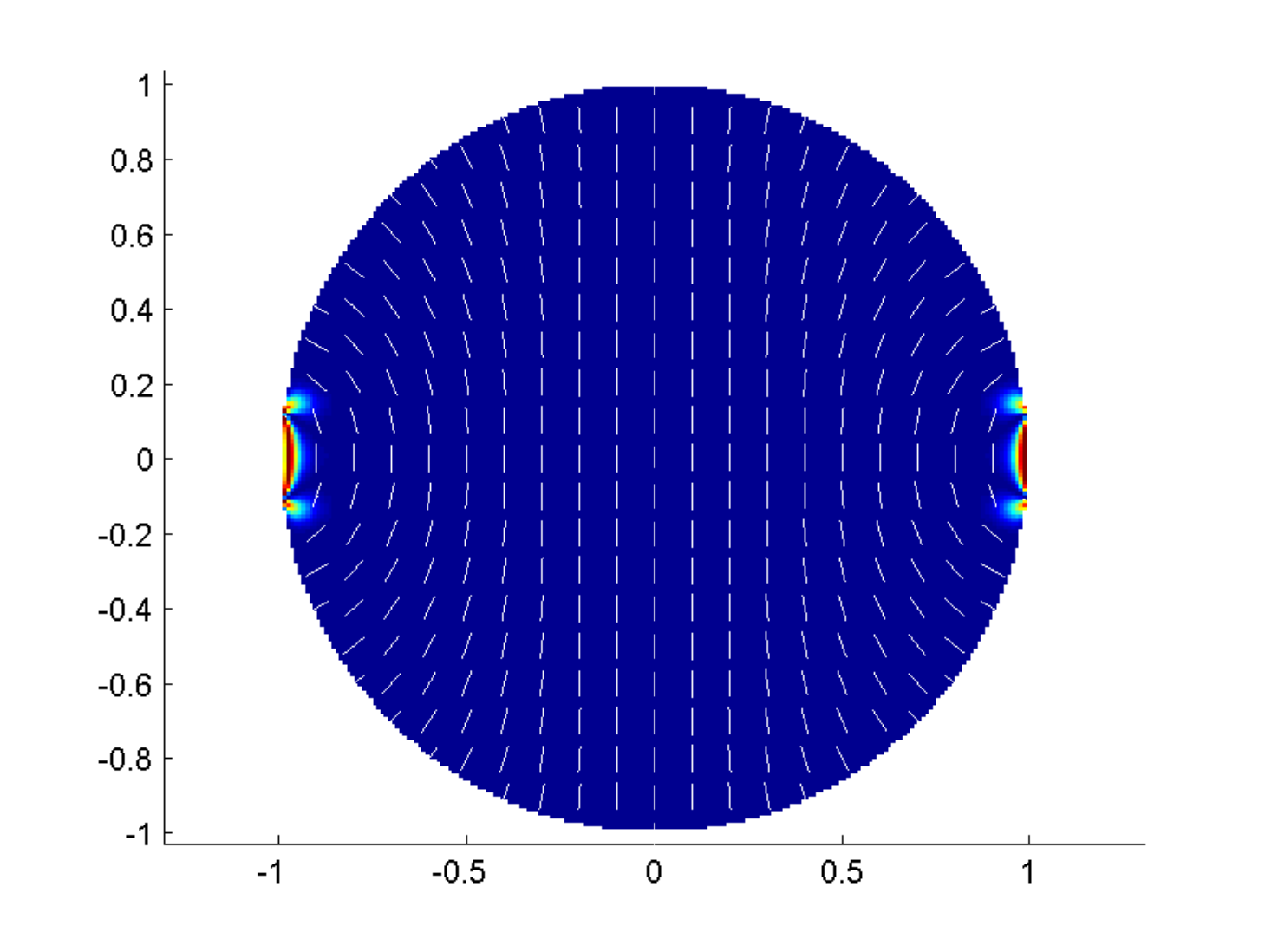}
}
\caption{Three axisymmetric defect patterns, shown by the slice of $x_2$-$x_3$ plane, where $x_3$ is the axis of symmetry. White rods represent principal eigenvectors. The background color describes the biaxiality $\mu$, with red indicates biaxial and blue indicates uniaxial. In all three cases $\alpha_2=0.04$, and $\alpha_1$ are chosen as: (a) $\alpha_1=11$; (b) $\alpha_1=16$; (c) $\alpha_1=22$.}
\label{LCPhase}
\end{figure}

We fix $\alpha_2=0.04$ and let $\alpha_1$ vary.
Three defect patterns are observed and drawn in Figure \ref{LCPhase}:
radial hedgehog (Figure \ref{LCPhase:a}), when $\alpha_1=11$;
ring disclination (Figure \ref{LCPhase:b}), when $\alpha_1=16$;
sphere ring band (Figure \ref{LCPhase:c}), when $\alpha_1=22$.
In the radial hedgehog pattern, $Q$ is uniaxial everywhere
with the principal eigenvector along the radial direction. The sphere center,
where $Q=0$, is the only point defect.
In the ring disclination pattern, the points where $Q$ has two identical positive eigenvalues form a circle in the $x$-$y$ plane, round which is a torus of biaxial region. In the sphere ring band pattern, the points where $Q=0$ form two rings on the spherical surface. In the band between these two rings on the spherical surface, $Q$ has two identical positive eigenvalues. A strong biaxial region is observed inside the sphere near the band. The last pattern is not found in the Landau-de Gennes theory \cite{hu2016disclination}. We believe that this novel pattern come from the term $B:(Q+I/3)-\log Z$, since it is the only term different from the Landau-de Gennes theory. Hence, it is necessary for this model to be further examined.

\section{Conclusion\label{concl}}

We develop a fast and accurate algorithm to evaluate the moments of Bingham distribution. Numerical test shows that it is remarkbly faster than direct numerical quadrature, while maintaining high accuracy. We apply the algorithm to the liquid crystal model that contains the Bingham distribution, which is able to constrain the order parameters within the physical range. We examine the defect patterns of liquid cystals confined inside a sphere and find a novel pattern,
suggesting that the model be examined thoroughly and compared with the Landau-de Gennes theory in future studies.
Armed with our algorithm, these studies will become much less expensive computationally.

\vspace{12pt}
\textbf{Acknowledgment} Pingwen Zhang is supported by National Natural Science Foundations of China (Grant No. 11421101 and No. 11421110001).

\bibliographystyle{plain}
\bibliography{references}
\end{document}